\documentclass[11 pt]{amsart}
\usepackage{amscd,amsfonts,amssymb,amsmath}
\usepackage{hyperref}

\usepackage{tikz}
\usepackage[margin=3.7  cm]{geometry}

\newtheorem{theorem}{Theorem}[section]
\newtheorem{cor}[theorem]{Corollary}
\newtheorem{lemma}[theorem]{Lemma}
\newtheorem{prop}[theorem]{Proposition}

\theoremstyle{definition}

\numberwithin{equation}{subsection}
\theoremstyle{plain}

\newtheorem{problem}{Problem}

\def \Z{\mathbb Z}

\newcommand{\secref}[1]{Section~\ref{#1}}
\newcommand{\thmref}[1]{Theorem~\ref{#1}}
\newcommand{\lemref}[1]{Lemma~\ref{#1}}

\newcommand{\corref}[1]{Corollary~\ref{#1}}

\newcommand{\eqnref}[1]{~{\textrm(\ref{#1})}}
\numberwithin{equation}{section}

\begin{document}
\title[Commutator Subgroups of Virtual and Welded Braid Groups]{Commutator Subgroups of Virtual and Welded Braid Groups}
\author[V. G. Bardakov]{Valeriy G. Bardakov}
\author[K. Gongopadhyay]{Krishnendu Gongopadhyay}
\author[M. V. Neshchadim]{ Mikhail V. Neshchadim}
\address{Sobolev Institute of Mathematics and Novosibirsk State University, Novosibirsk 630090, Russia.}
\address{Department of AOI, Novosibirsk State Agrarian University, Dobrolyubova street, 160, Novosibirsk, 630039, Russia.}

\email{bardakov@math.nsc.ru}
\address{Indian Institute of Science Education and Research (IISER) Mohali, Sector 81,  S. A. S. Nagar, P. O. Manauli, Punjab 140306, India.}
\email{krishnendu@iisermohali.ac.in}
\address{Sobolev Institute of Mathematics and Novosibirsk State University, Novosibirsk 630090, Russia.}
\email{neshch@math.nsc.ru}
\thanks{The authors  acknowledge partial support from grant of Russian Science Foundation, project N~16-41-02006 (from the DST project INT/RUS/RSF/P-2).}
\date{\today}

\begin{abstract}
Let $VB_n$, resp. $WB_n$ denote the virtual, resp. welded, braid group on $n$ strands. We study their commutator subgroups $VB_n' = [VB_n, VB_n]$ and, $WB_n' = [WB_n, WB_n]$ respectively.  We obtain a set of generators and defining relations for these commutator subgroups. In particular, we prove that $VB_n'$ is finitely generated if and only if $n \geq 4$, and  $WB_n'$ is finitely generated for $n \geq 3$.  Also we prove that $VB_3'/VB_3'' =\mathbb{Z}_3 \oplus \mathbb{Z}_3 \oplus\mathbb{Z}_3 \oplus \mathbb{Z}^{\infty}$,
$VB_4' / VB_4'' = \mathbb{Z}_3 \oplus \mathbb{Z}_3 \oplus \mathbb{Z}_3$,  $WB_3'/WB_3'' =
\mathbb{Z}_3 \oplus \mathbb{Z}_3 \oplus\mathbb{Z}_3 \oplus \mathbb{Z},$  $WB_4'/WB_4'' =
\mathbb{Z}_3,$ and for $n \geq 5$ the commutator subgroups $VB_n'$ and $WB_n'$ are perfect, i.~e. the commutator subgroup is equal to the second commutator subgroup.
\end{abstract}
\maketitle

\section{Introduction}

Virtual braid groups $VB_n$ on $n$ strands are certain extensions of the classical braid groups. It was introduced by L.~Kauffman \cite{lk1} (see also \cite{ve}). Virtual braids play the same role in the virtual knot theory that classical braids played in the classical knot theory. In particular, like closures of classical braids represent classical knots and links, the closure of virtual braids represent the virtual knots and links (see  \cite{kamada}, \cite{lk1}). On connections of virtual braids with the virtual knot theory, see \cite{BMN,BMN-1}. For a structure of the virtual braid groups, see \cite{B}.

The welded braid group $WB_n$ is a quotient of $VB_n$. This group is called the group of conjugating automorphisms \cite{Sav, Bar}, the braid-permutation group \cite{FRR} and so on. For several notions of this group and their equivalence, see \cite{dam}.

  The commutator subgroup $B_n'$ of the classical braid group $B_n$ is studied in the paper \cite{gl} (see also \cite{sa}). The following facts follow from these papers:

\medskip

--- $B_n'$ is finitely presented for all $n \geq 2$;

--- $B_3'$ is a free group of rank two;

--- $B_4'$ is a semi-direct product of two free groups of rank two;

--- for $n>4$ the second commutator subgroup $B_n''$ of $B_n$
coincides with the first commutator subgroup  $B_n'$, i.e.  $B_n'$ is perfect.

\medskip

In the present paper we investigate the commutator subgroups $VB_n'$  and $WB_n'$.  Our main result is the following.

\begin{theorem}\label{mainth}
The commutator subgroup $VB_3'$ is infinitely generated.
For $n \geq 4$ the commutator subgroup $VB_n'$ can be generated by $2n - 3$ elements.
\end{theorem}

\medskip
To prove \thmref{mainth} we obtain a presentation of $VB_n'$ using the classical method of Reidemeister-Schreier,  and then remove certain generators and relations using Tietze transformations. As a consequence of Theorem \ref{mainth}, we further have the following corollaries.

\begin{cor}\label{cor1}
\begin{enumerate} \item
	The quotient  $VB_3'/VB_3''$ is isomorphic to the direct product
$
\mathbb{Z}_3 \oplus \mathbb{Z}_3 \oplus\mathbb{Z}_3 \oplus \mathbb{Z}^{\infty},
$
where $\mathbb{Z}^{\infty}$ is the direct product of counting number of $\mathbb{Z}$.

 \item
	The quotient  $VB_4'/VB_4''$ is isomorphic to the direct product
	$\mathbb{Z}_3 \oplus  \mathbb{Z}_3 \oplus\mathbb{Z}_3$.

 \item	For $n\geq 5$, $VB_n'$ is perfect, that is  $VB_n'=VB_n''$.\end{enumerate}
\end{cor}

\medskip Recently the commutator subgroup $WB_n'$ of the welded braid group has been investigated by Zaremsky in \cite{mz}, who proved that $WB_n'$ is finitely presented if and only if $n \geq 4$.  Zaremsky proved this result using discrete Morse theory,  without constructing explicit finite presentation. Dey and Gongopadhyay \cite{dg} also proved that $WB_n'$ is finitely generated for all $n \geq 3$.  In the present paper we have found a better bound on the number of generators than  in \cite{dg}.  We prove the following result.

 \begin{theorem}\label{wbth}
\begin{enumerate}

\item
	The commutator subgroup $WB_n'$ can be generated by $n$ elements  for all $n \geq 4$, and $WB_3'$ can be generated by $4$ elements.

\item
The quotient  $WB_3'/WB_3''$ is isomorphic to the direct product
	$\mathbb{Z}_3 \oplus  \mathbb{Z}_3 \oplus\mathbb{Z}_3 \oplus \mathbb{Z}$.

\item
The quotient  $WB_4'/WB_4''$ is isomorphic to
	$\mathbb{Z}_3$.

\item	For $n\geq 5$, $WB_n'$ is perfect.
\end{enumerate}
\end{theorem}

\medskip The presentation of $WB_n'$  obtained in this paper is  slightly different from the one obtained in \cite{dg}.  We obtain this presentation using the presentation of $VB_n$, while computing that we use successive conjugation rule in the rewritting process, see \lemref{l1}. This simple conjugation tirck has given an alternative presentation of $WB_n'$ where the elimination of generators become simpler, and consequently we get a better bound on the number of generators.

\bigskip

We now briefly describe the structure of the paper. We recall the necessary preliminaries in \secref{prel}.   Using Reidemeister-Schreier method, we first obtain a general presentation of $VB_n'$, see \thmref{nsp} in \secref{comv}.  We prove  \thmref{mainth} in \secref{pf}.  Because of the differences of the nature of the proofs for  $n\geq 4$ and $n=3$,   As the cases $n=3$ and $n \geq 4$ are different, accordingly, the proof of \thmref{mainth} is divided over two subsections. In \secref{vb4}, first, we apply Tietze transformations to remove certain generators from the presentation in \thmref{nsp}. This gives a  finite generating set for $VB_n'$ for $n \geq 4$.  In \secref{vb3},  we show that $VB_3'$ is infinitely generated. Combining these results, \thmref{mainth} is obtained. We prove \thmref{wbth} in \secref{wel}.

\bigskip Finally, we note  the following problems that remain to be answered.

\begin{problem}
Is it true that the commutator subgroup $VB_n'$ is not finitely presented for  $n \geq 4$.
\end{problem}

We expect the answer to be yes, but it is not clear how.

\begin{problem}
Construct explicit finite presentation of $WB_n'$ for $n \geq 3$.
\end{problem}

\begin{problem}
Let $G$ is a group from the set $\{ VB_3, VB_4, WB_3, WB_4 \}$. Find the quotients $G^{(i)} / G^{(i+1)}$, $i = 2, 3, \ldots$, where $G^{(k)}$ is the $k$-th commutator subgroup:
$$
G^{(1)} = G',~~G^{(k+1)} = [G^{(k)}, G^{(k)}], ~~~k = 1, 2, \ldots.
$$
\end{problem}

\section{Preliminaries} \label{prel}
\subsection{Group of Virtual Braids}
The virtual braid group of $n$ strands $VB_{n}$  is generated by the classical braid group
$B_n = \langle \sigma_1,$ $\ldots,$ $\sigma_{n-1} \rangle$
and the symmetric group $S_n = \langle\rho_1, \ldots,\rho_{n-1}\rangle$.
The generators $\sigma_{i}, i = 1, \ldots, n-1$
satisfy the relations
$$
\begin{array}{ccc}
  \sigma_i \sigma_j = \sigma_j \sigma_i                             & \,\, \mbox{for} \,\, & |i-j|\geq 2, \\
  \sigma_i \sigma_{i+1} \sigma_i= \sigma_{i+1} \sigma_i \sigma_{i+1}& \,\, \mbox{for} \,\, & i=1,...,n-2. \\

\end{array}
$$
The generators $\rho_{i}, i = 1, \ldots, n-1$  satisfy the relations of symmetric group
 $S_n$:
$$
\begin{array}{ccc}
 \rho_{i}^{2} = 1 & \,\, \mbox{for} \,\, & i = 1, 2, \ldots,n-1,\\
 \rho_{i}\rho_{j} = \rho_{j}\rho_{i} & \,\, \mbox{for} \,\, & |i-j|\geq 2,\\
 \rho_{i}\rho_{i+1}\rho_{i} = \rho_{i+1}\rho_{i}\rho_{i+1}& \,\, \mbox{for} \,\, & i = 1, 2 \ldots,n-2.\\
 \end{array}
$$
Other defining relations of $VB_{n}$ are mixed and have the form
$$
\begin{array}{ccc}
\sigma_{i}\rho_{j}= \rho_{j}\sigma_{i} & \,\, \mbox{for} \,\, & |i-j|\geq 2,\\
\rho_{i}\rho_{i+1}\sigma_{i} = \sigma_{i+1}\rho_{i}\rho_{i+1}& \,\, \mbox{for} \,\, & i = 1,2,\ldots,n-2.
\end{array}
$$

\medskip
\subsection{Reidemeister-Schreier Algorithm} Given a presentation of a group $G$, this algorithm allows one to find a presentation of a subgroup $H\subset G$. To obtain the presentation of $H$,   it is necessary to find a  Schreier's set of right coset of the group $G$ over the subgroup $H$. We give a formal description of this process, for more details see \cite{mks}.

\medskip
Let $a_1,\ldots,a_n$ be the generators of the group $G$ and $R_1,\ldots,R_m$ be
the set of defining relations for the given set of generators.
System of words $N=\left\{ K_\alpha,\, \alpha\in A  \right\}$
on generators $a_1,\ldots,a_n$  defines a Schreier's system for the subgroup $H\subset G$
relative to the system of generators $a_1,\ldots,a_n$
if the next conditions are satisfied:

\medskip 1) in every right coset of the group $G$ over $H$
there is only one word from the system $N$;

\medskip 2) if the word
$K_\alpha=a_{i_1}^{\varepsilon_1}\ldots a_{i_{p-1}}^{\varepsilon_{p-1}}a_{i_p}^{\varepsilon_p}$,
$(\varepsilon_j=\pm 1)$ lies in $N$,
then the word $a_{i_1}^{\varepsilon_1}\ldots a_{i_{p-1}}^{\varepsilon_{p-1}}$
also lies in $N$.

\medskip Suppose that some
Schreier's system $N$ is chosen for the subgroup $H\subset G$
relative to the system generators $a_1,\ldots,a_n$ of $G$.
For every word $Q$ on $a_1,\ldots,a_n$, we denote by $\overline{Q}$
the only word from $N$ which lies in the same right coset of $G$ over the subgroup $H$.
Denote
$$
S_{K_\alpha, a_\nu}=K_\alpha a_\nu \cdot (\overline{K_\alpha a_\nu})^{-1},
\quad \alpha\in A,\,\,\nu=1,\ldots,n.
$$
Theorem of Reidemeister-Schreier states that the elements $S_{K_\alpha, a_\nu}$
generate subgroup $H$ and
the set of defining relations for this set of generators
is divided in two parts.
First part consists of trivial relations $S_{K_\alpha, a_\nu}=1$,
where the pair $K_\alpha$, $a_\nu$ is such that the word
$K_\alpha a_\nu \cdot (\overline{K_\alpha a_\nu})^{-1}$
is freely equivalent to the word 1.
Second part consists of all relations of the form $\tau(K_\alpha R_\mu K_\alpha^{-1})$,
where $\alpha\in A$, $\mu=1,\ldots,m$, and $\tau$ is Reidemeister's transformation,
which  maps every nonempty word
$a_{i_1}^{\varepsilon_1}\ldots a_{i_p}^{\varepsilon_p}$, $(\varepsilon_j=\pm 1)$
from symbols $a_1,\ldots,a_n$ to the word from symbols $S_{K_\alpha, a_\nu}$
by the rule:
$$
\tau (a_{i_1}^{\varepsilon_1}\ldots a_{i_p}^{\varepsilon_p})=
 S_{K_{i_1}, a_{i_1}}^{\varepsilon_1} \ldots S_{K_{i_p}, a_{i_p}}^{\varepsilon_p},
$$
where $K_{i_j}=\overline{a_{i_1}^{\varepsilon_1}\ldots a_{i_{j-1}}^{\varepsilon_{j-1}}}$,
if $\varepsilon_j=1$, and
$K_{i_j}=\overline{a_{i_1}^{\varepsilon_1}\ldots a_{i_{j}}^{\varepsilon_{j}}}$,
if $\varepsilon_j=-1$.

\bigskip

\section{Commutator subgroup $VB_n'$} \label{comv}

\subsection{Generating set of $VB_n'$} \label{gen}
From the above relations it follows that the quotient  $VB_n/VB_n'$ is isomorphic to the direct product
$\mathbb{Z}\times \mathbb{Z}_2$. One can define the map $\varphi$ from the following short exact sequence:
\begin{equation*}\label{se1}1 \xrightarrow {} VB_n' \xrightarrow{} VB_n \xrightarrow{\varphi} \Z \times \Z_2 \xrightarrow{} 1\end{equation*}
where, for $i=1, \ldots, n-1$, $\varphi(\sigma_i)$ is the generator of $\Z$ and  $\phi(\rho_i)$ is the generator of  $\Z_2$ respectively when viewing it as $VB_n/VB_n'$.  The map $\varphi$ does have a section in the above short exact sequence for $n \geq 3$,  and  $\ker \varphi = VB_n'$.

\medskip
As a Schreier set of coset representatives of  $VB_n$ by $VB_n'$ take the words
$$
\Lambda=\left\{\,  \sigma_1^i\rho_1^\varepsilon \,|\, i\in
\mathbb{Z},\,\, \varepsilon=0,1 \,\right\}.
$$

\medskip The commutator subgroup $VB_n'$ is generated by the words
$$
S_{\lambda,a}=\lambda a (\overline{\lambda a})^{-1},\quad \lambda
\in \Lambda,\quad
 a \in \{ \sigma_1, \ldots, \sigma_{n-1}, \rho_1, \ldots, \rho_{n-1} \}.
$$
Here  $\overline{w}$ is a coset representative in $\Lambda$ of
$wVB_n'$. Find the elements $S_{\lambda,a}$. For this put $\lambda = \sigma_1^i\rho_1^\varepsilon$ and considering different $a$ we will get the following cases:

1)  If $a=\sigma_1$, then
$$
S_{\lambda,\sigma_1}=\sigma_1^i\rho_1^\varepsilon \sigma_1 (\sigma_1^{i+1}\rho_1^\varepsilon)^{-1}.
$$
For  $\varepsilon=0$ we have $S_{\lambda,\sigma_1}=1$
and for $\varepsilon=1$ we have
$S_{\lambda,\sigma_1}=\sigma_1^i (\rho_1 \sigma_1 \rho_1 \sigma_1^{-1} ) \sigma_1^{-i}$, which we will denote by $a_i$.

\medskip 2) If $a=\sigma_2$, then
$$
S_{\lambda,\sigma_2}=
\sigma_1^i (\rho_1^\varepsilon\sigma_2\rho_1^\varepsilon\sigma_1^{-1})\sigma_1^{-i},
$$
and we will denote this element by $b_{i,\varepsilon}$.

\medskip  3)  If $a=\sigma_l$, $l>2$, then
$$
S_{\lambda,\sigma_l}=\sigma_l\sigma_1^{-1},
$$
and we will denote this element by $c_l$.

\medskip  4) If $a=\rho_1$, then
$$
S_{\lambda,\rho_1}=1.
$$

\medskip  5) If  $a=\rho_2$, then
$$
S_{\lambda,\rho_2}=
\sigma_1^i (\rho_1^\varepsilon\rho_2\rho_1^{\varepsilon+1})\sigma_1^{-i},
$$
and we will denote this element by $f_{i,\varepsilon}$.

\medskip  6) If  $a=\rho_l$, $l>2$, then
$$
S_{\lambda,\rho_l}=\sigma_1^i (\rho_l\rho_1)\sigma_1^{-i},
$$
and we will denote this element by $g_{i,l}$.

To find defining relations of $VB_n'$ we will use the following  conjugation rules by elements $\rho_1$ and $\sigma_1^{-m}$.


\begin{lemma} \label{l1}
The following formulas hold
\begin{itemize}
\item[(1)] $a_i^{\sigma_1^{-m}}=a_{i+m}, \quad
b_{i,\varepsilon}^{\sigma_1^{-m}}=b_{i+m,\varepsilon}, \quad
c_l^{\sigma_1^{-m}}=c_l, \quad
f_{i,\varepsilon}^{\sigma_1^{-m}}=f_{i+m,\varepsilon}, \quad
g_{i,\varepsilon}^{\sigma_1^{-m}}=g_{i+m,\varepsilon};$

\item[(2)] $a_0^{\rho_1}=a_0^{-1}, \quad b_{0,0}^{\rho_1}=b_{0,1} a_0^{-1}, \quad
b_{0,1}^{\rho_1}=b_{0,0}a_0^{-1}, \quad b_{1,0}^{\rho_1}=a_0 b_{1,1} a_1^{-1} a_0^{-1}, \quad b_{2,0}^{\rho_1}=a_0 a_1 (b_{2,1} a_2^{-1}) a_1^{-1} a_0^{-1};$

\item[(3)] $c_l^{\rho_1}=c_l a_0^{-1}, \quad f_{0,0}^{\rho_1}=f_{0,1}, \quad f_{0,1}^{\rho_1}=f_{0,0}, \quad f_{1,0}^{\rho_1}=a_{0}f_{1,1}a_0^{-1}, \quad
f_{1,1}^{\rho_1}=a_{0}f_{1,0}a_0^{-1};$

\item[(4)] $g_{0,i}^{\rho_1}=g_{0,i}, \quad g_{1,i}^{\rho_1}=a_{0}g_{1,i}a_0^{-1}, \quad i > 2.$
\end{itemize}
\end{lemma}

\begin{proof} (1) follow from the definition.

For proving (2) note that:
$$
\rho_1 a_0 \rho_1 = \rho_1 \rho_1 \sigma_1 \rho_1 \sigma_1^{-1} \rho_1=
S_{1,\sigma_1}S_{\sigma_1,\rho_1}S_{\rho_1,\sigma_1}^{-1}S_{\rho_1,\rho_1}=a_0^{-1},
$$
$$
\rho_1 b_{0,0} \rho_1 =\rho_1 \sigma_2 \sigma_1^{-1} \rho_1=
S_{1,\rho_1} S_{\rho_1,\sigma_2} S_{\rho_1,\sigma_1}^{-1}S_{\rho_1,\rho_1}=b_{0,1} a_0^{-1},
$$
$$
b_{0,1}^{\rho_1}=\rho_1 \rho_1 \sigma_2 \rho_1 \sigma_1^{-1} \rho_1=
S_{1,\sigma_2}S_{\sigma_1,\rho_1}S_{\rho_1,\sigma_1}^{-1}S_{\rho_1,\rho_1}
=b_{0,0}a_0^{-1},
$$
$$
\rho_1 b_{1,0} \rho_1 =\rho_1 \sigma_1\sigma_2 \sigma_1^{-1}\sigma_1^{-1} \rho_1=
S_{1,\rho_1} S_{\rho_1,\sigma_1} S_{\sigma_1\rho_1,\sigma_2}
S_{\sigma_1\rho_1,\sigma_1}^{-1}S_{\rho_1,\sigma_1}^{-1}=a_0 b_{1,1} a_1^{-1} a_0^{-1}.
$$
$$
\rho_1 b_{2,0} \rho_1 =\rho_1  \sigma_1^{2} \sigma_2 (\sigma_1^{-1})^3 \rho_1=
$$
$$
=S_{1,\rho_1} S_{\rho_1,\sigma_1} S_{\sigma_1\rho_1,\sigma_1}
S_{\sigma_1^2\rho_1,\sigma_2}S_{\sigma_1^2\rho_1,\sigma_1}^{-1} S_{\sigma_1\rho_1,\sigma_1}^{-1}
S_{\rho_1,\sigma_1}^{-1}S_{\rho_1,\rho_1}=a_0 a_1 (b_{2,1} a_2^{-1}) a_1^{-1} a_0^{-1}.
$$

For (3):
$$
\rho_1 c_l \rho_1 =\rho_1 \sigma_l \sigma_1^{-1} \rho_1=
S_{1,\rho_1} S_{\rho_1,\sigma_l} S_{\rho_1,\sigma_1}^{-1}S_{\rho_1,\rho_1}=c_l a_0^{-1},
$$
$$
f_{0,0}^{\rho_1}=\rho_1 \rho_2\rho_1\rho_1=\rho_1 \rho_2=f_{0,1},
$$
$$
f_{0,1}^{\rho_1}= \rho_1\rho_1 \rho_2\rho_1= \rho_2 \rho_1= f_{0,0},
$$
$$
f_{1,0}^{\rho_1}=\rho_1  \sigma_1 \rho_2 \rho_1 \sigma_1^{-1} \rho_1=
S_{1,\rho_1}S_{\rho_1,\sigma_1}S_{\sigma_1\rho_1,\rho_2}S_{\sigma_1,\rho_1}
S_{\rho_1,\sigma_1}^{-1}S_{\rho_1,\rho_1}
=a_{0}f_{1,1}a_0^{-1}.
$$
$$
f_{1,1}^{\rho_1}=
\rho_1  \sigma_1 \rho_1 \rho_{2} \sigma_{1}^{-1}\rho_1 =
S_{1,\rho_1} S_{\rho_1,\sigma_1}S_{\sigma_1\rho_1,\rho_1} S_{\sigma_1,\rho_2}
S_{\rho_1,\sigma_1}^{-1}S_{\rho_1,\rho_1}
=a_{0}f_{1,0}a_0^{-1}.
$$

(4):
$$
g_{0,i}^{\rho_1}=  \rho_1 \rho_i\rho_1\rho_1 =\rho_1 \rho_i=g_{0,i},
$$
$$
g_{1,i}^{\rho_1}=\rho_1 \sigma_1 \rho_i \rho_1 \sigma_1^{-1} \rho_1=
S_{1,\rho_1} S_{\rho_1,\sigma_1} S_{\sigma_1\rho_1,\rho_i}S_{\sigma_1,\rho_1}
S_{\rho_1,\sigma_1}^{-1}S_{\rho_1,\rho_1}=a_{0}g_{1,i}a_0^{-1},  \quad i > 2.
$$
This proves the lemma.
\end{proof}


\subsection{Defining Relations in $VB_n'$} In this subsection we will consider the defining relations of $VB_n$, rewrite them in the generators of $VB_n'$, and conjugating by elements $\lambda \in \Lambda$, we get the defining relations of $VB_n'$.

\subsubsection{Defining relation of $VB_n'$ that follow from the relation $\sigma_i \sigma_j  = \sigma_j \sigma_i$} Rewrite this relation in the form
$$
r_1=\sigma_i \sigma_j  \sigma_i^{-1} \sigma_j^{-1}, ~1\leq i< j\leq n-1, ~i+1<j.
$$
Using the rewritting process we get:

for $i=1$:
$r_1=\sigma_1 \sigma_j  \sigma_1^{-1} \sigma_j^{-1}=
S_{1,\sigma_1} S_{\sigma_1,\sigma_j} S_{\sigma_1,\sigma_1}^{-1}S_{\sigma_1,\sigma_j}^{-1}=c_j c_j^{-1}=1;$

\medskip for $i=2$:
$r_1=\sigma_2 \sigma_j  \sigma_2^{-1} \sigma_j^{-1}=
S_{1,\sigma_2} S_{\sigma_1,\sigma_j} S_{\sigma_1,\sigma_2}^{-1}S_{\sigma_1,\sigma_j}^{-1}=
b_{0,0}c_j b_{1,0}^{-1}c_j^{-1};$

\medskip for  $i>2$:
$r_1=\sigma_i \sigma_j  \sigma_i^{-1} \sigma_j^{-1}=
S_{1,\sigma_i} S_{\sigma_1,\sigma_j} S_{\sigma_1,\sigma_i}^{-1}S_{\sigma_1,\sigma_j}^{-1}=
c_ic_j c_i^{-1}c_j^{-1}$.


\begin{lemma} \label{l3}
The following four types of relations in $VB_n'$ follow from the relation $r_1$ of $VB_n$:
$$
b_{m,0} c_j b_{m+1,0}^{-1} c_j^{-1} = 1, \quad j \geq 4,
$$
$$
c_i c_j c_i^{-1} c_j^{-1} = 1, \quad i \geq 3,  \quad j > i+1
$$
$$
b_{m,1} a_m^{-1} c_j a_{m+1} b_{m+1,1}^{-1} c_j^{-1} = 1, \quad j \geq 4,
$$
$$
c_i a_m^{-1} c_j c_i^{-1} a_m c_j^{-1} = 1, \quad i \geq 3, \quad j > i+1.
$$
\end{lemma}

\begin{proof} Conjugating of $r_1$ by $\rho_1$ and using  Lemma \ref{l1} we get
$$
\rho_1 (b_{0,0}c_j b_{1,0}^{-1}c_j^{-1}) \rho_1 = b_{0,1}a_0^{-1}c_j a_{1}b_{1,1}^{-1}c_j^{-1},
$$
$$
\rho_1 (c_ic_j c_i^{-1}c_j^{-1}) \rho_1=c_ia_0^{-1}c_j c_i^{-1}a_0^{-1}c_j^{-1}.
$$
Conjugating $r_1$ and $\rho_1 r_1 \rho_1$ by $\sigma_1^{-m}$, we get the need relations.
\end{proof}


\subsubsection{Defining relations of $VB_n'$ that follow from the relation $\sigma_i \sigma_{i+1} \sigma_i = \sigma_{i+1} \sigma_i \sigma_{i+1}$} $\;$
Rewrite this relation in the form
$$
r_2=\sigma_i \sigma_{i+1} \sigma_i \sigma_{i+1}^{-1} \sigma_i^{-1} \sigma_{i+1}^{-1}.
$$

Then

for $i=1$:
$r_2=\sigma_1 \sigma_{2} \sigma_1 \sigma_{2}^{-1} \sigma_1^{-1} \sigma_{2}^{-1}=
S_{1,\sigma_1} S_{\sigma_1,\sigma_2} S_{\sigma_1^2,\sigma_1} S_{\sigma_1^2,\sigma_1}^{-1}
S_{\sigma_1,\sigma_1}^{-1}S_{1,\sigma_2}^{-1}=b_{1,0}b_{2,0}^{-1}b_{0,0}^{-1}$.

\medskip for $i=2$:
$r_2=\sigma_2 \sigma_{3} \sigma_2 \sigma_{3}^{-1} \sigma_2^{-1} \sigma_{3}^{-1}=
S_{1,\sigma_2} S_{\sigma_1,\sigma_3} S_{\sigma_1^2,\sigma_2} S_{\sigma_1^2,\sigma_3}^{-1}
S_{\sigma_1,\sigma_2}^{-1}S_{1,\sigma_3}^{-1}=b_{0,0}c_3b_{2,0}c_3^{-1}b_{1,0}^{-1}c_3^{-1}$.

\medskip
for $i>2$:
$r_2=\sigma_i \sigma_{i+1} \sigma_i \sigma_{i+1}^{-1} \sigma_i^{-1} \sigma_{i+1}^{-1}=
S_{1,\sigma_i} S_{\sigma_1,\sigma_{i+1}} S_{\sigma_1^2,\sigma_i} S_{\sigma_1^2,\sigma_{i+1}}^{-1}
S_{\sigma_1,\sigma_i}^{-1}S_{1,\sigma_{i+1}}^{-1}=c_i c_{i+1} c_i c_{i+1}^{-1} c_i^{-1} c_{i+1}^{-1}$.

\begin{lemma} \label{l5}
The following six types of relations in $VB_n'$ follow from the relation $r_2$ of $VB_n$:
$$
b_{m+1,0}b_{m+2,0}^{-1}b_{m,0}^{-1}=1,
$$
$$
b_{m,0}c_3b_{m+2,0}c_3^{-1}b_{m+1,0}^{-1}c_3^{-1}=1,
$$
$$
c_i c_{i+1} c_i c_{i+1}^{-1} c_i^{-1} c_{i+1}^{-1}=1, \quad i \geq 3,
$$
$$
a_mb_{m+1,1}a_{m+2}b_{m+2,1}^{-1}a_{m+1}^{-1}b_{m,1}^{-1}=1,
$$
$$
b_{m,1}a_m^{-1}c_3 a_{m+1}b_{m+2,1}a_{m+2}^{-1}a_{m+1}^{-1} c_3^{-1}a_m a_{m+1} b_{m+1,1}^{-1}c_3^{-1}=1,
$$
$$
c_i a_m^{-1}c_{i+1} a_m^{-1}c_i c_{i+1}^{-1} a_m c_i^{-1} a_m c_{i+1}^{-1}=1,\quad i \geq 3.
$$
\end{lemma}

\begin{proof}
Conjugating  $r_2$ by $\rho_1$
and using Lemma \ref{l1}, we get
$$
(b_{1,0}b_{2,0}^{-1}b_{0,0}^{-1})^{\rho_1}=
a_0b_{1,1}a_{2}b_{2,1}^{-1}a_{1}^{-1}b_{0,1}^{-1},
$$
$$
(b_{0,0}c_3b_{2,0}c_3^{-1}b_{1,0}^{-1}c_3^{-1})^{\rho_1}=
b_{0,1}a_0^{-1}c_3 a_{1}b_{2,1}a_{2}^{-1}a_{1}^{-1} c_3^{-1}a_0 a_{1} b_{1,1}^{-1}c_3^{-1},
$$
$$
(c_i c_{i+1} c_i c_{i+1}^{-1} c_i^{-1} c_{i+1}^{-1})^{\rho_1}=
c_i a_0^{-1}c_{i+1} a_0^{-1}c_i c_{i+1}^{-1} a_0 c_i^{-1} a_0 c_{i+1}^{-1}.
$$
Conjugating $r_2$ and $\rho_1 r_2 \rho_1$ by $\sigma_1^{-m}$, we get the need relations.
\end{proof}

\subsubsection{Defining relation that follow from the relation $\rho_i^2 = 1$}

Rewrite this relation in the form
$$ r_3=\rho_i^2.$$

Then

for  $i=1$: $r_3=\rho_1 \rho_1 =S_{1,\rho_1} S_{\rho_1,\rho_1}=1$.

\medskip for $i=2$: $r_3=\rho_3 \rho_3 =S_{1,\rho_2} S_{\rho_1,\rho_2}=f_{0,0}f_{0,1}$.

\medskip for $i>2$:
$r_3=\rho_i \rho_i =S_{1,\rho_i} S_{\rho_1,\rho_i}=g_{0,i}^2$.

The following lemma holds

\begin{lemma} \label{l7}
From  $r_3$ follow two types of relations in $VB_n'$:
$$
f_{m,0}f_{m,1}=g_{m,i}^2=1,  \quad i>2.
$$
\end{lemma}

\begin{proof} Conjugating  $r_3$ by $\rho_1$
and using Lemma \ref{l1} 4, we get
$$
(f_{0,0}f_{0,1})^{\rho_1}=f_{0,1}f_{0,0},
$$
$$
(g_{0,i}^2)^{\rho_1}=g_{0,i}^2, \quad i>2.
$$
Conjugating $r_3$ and $\rho_1 r_3 \rho_1$ by $\sigma_1^{-m}$, we get the need relations.
\end{proof}

\subsubsection{Defining relations of $VB_n'$  that follow from the relation $\rho_i \rho_j = \rho_j \rho_i$}

Rewrite this relation in the form
$$r_4=\rho_i \rho_j  \rho_i \rho_j ,~ 1\leq i< j\leq n-1, ~i+1<j.$$

Then

for  $i=1$:
$r_4=\rho_1 \rho_j  \rho_1 \rho_j=
S_{1,\rho_1} S_{\rho_1,\rho_j} S_{1,\rho_1} S_{\rho_1,\rho_j} =g_{0,j}^{2}$,  $j>2$.

We have got this relation when we considered the relation  $r_3$.

\medskip for $i=2$:
$r_4=\rho_2 \rho_k  \rho_2 \rho_k=
S_{1,\rho_2} S_{\rho_1,\rho_k} S_{1,\rho_2} S_{\rho_1,\rho_k} =(f_{0,0}g_{0,k})^{2}$,  $k>3$.

\medskip for $i>2$:
$r_4=\rho_i \rho_j  \rho_i \rho_j=
S_{1,\rho_i} S_{\rho_1,\rho_j} S_{1,\rho_i} S_{\rho_1,\rho_j} =(g_{0,i}g_{0,j})^{2}$, $3\leq i< j\leq n-1$, $i+1<j$.

\medskip The following lemma holds

\begin{lemma} \label{l8}
From the relation  $r_4$  of $VB_n$, the following three types of relations of $VB_n'$ follow:
$$
(f_{m,0}g_{m,k})^2=(f_{m,1}g_{m,k})^{2}=1, \quad k > 3,
$$
$$
(g_{m,i}g_{m,j})^{2}=1, \quad 3\leq i< j\leq n-1, \quad i+1<j.
$$
\end{lemma}

\begin{proof} Conjugated  $r_4$ by $\rho_1$
and using Lemma \ref{l1} 4), we get
$$
\rho_1 (f_{0,0}g_{0,k})^{2} \rho_1=(f_{0,1}g_{0,k})^{2}, \quad k>3,
$$
$$
\rho_1 (g_{0,i}g_{0,j})^{2} \rho_1=(g_{0,i}g_{0,j})^{2}, \quad 3\leq i< j\leq n-1, \quad i+1<j.
$$
Conjugating $r_4$ and $\rho_1 r_4 \rho_1$ by $\sigma_1^{-m}$, we get the need relations.
\end{proof}

\subsubsection{Defining relations of $VB_n'$  that follow from the relation $\rho_i \rho_{i+1} \rho_i = \rho_{i+1} \rho_i \rho_{i+1}$}

Rewrite this relation in the form
$$r_5=\rho_i \rho_{i+1} \rho_i \rho_{i+1} \rho_i \rho_{i+1}.$$

Then

for  $i=1$:
$r_5=\rho_1 \rho_{2} \rho_1 \rho_{2} \rho_1 \rho_{2}=
S_{1,\rho_1} S_{\rho_1,\rho_2} S_{1,\rho_1} S_{\rho_1,\rho_2}S_{1,\rho_1} S_{\rho_1,\rho_2}=
f_{0,1}^{3}$.

\medskip
for  $i=2$:
$r_5=\rho_2 \rho_3 \rho_2 \rho_3 \rho_2 \rho_3=
S_{1,\rho_2} S_{\rho_1,\rho_3} S_{1,\rho_2} S_{\rho_1,\rho_3}S_{1,\rho_2} S_{\rho_1,\rho_3}=
(f_{0,0}g_{0,3})^{3}$.

\medskip
for $i>2$:
$r_5=\rho_i \rho_{i+1} \rho_i \rho_{i+1} \rho_i \rho_{i+1}=
S_{1,\rho_i} S_{\rho_1,\rho_{i+1}} S_{1,\rho_i} S_{\rho_1,\rho_{i+1}}S_{1,\rho_i} S_{\rho_1,\rho_{i+1}}=
(g_{0,i}g_{0,i+1})^{3}$.

The following lemma holds

\begin{lemma} \label{l8-1}
From the relation  $r_5$ of $VB_n$, we have the following five types of relations of $VB_n'$:
$$
f_{m,1}^{3}=1,
$$
$$
(f_{m,0}g_{m,3})^{3}=1,
$$
$$
(g_{m,i}g_{m,i+1})^{3}=1,
$$
$$
f_{m,0}^{3}=1,
$$
$$
(f_{m,1}g_{m,3})^{3}=1.
$$
\end{lemma}

\begin{proof} Conjugating  $r_5$ by $\rho_1$
and using Lemma \ref{l1} 4), we get
$$
(f_{0,1}^{3})^{\rho_1}=f_{0,0}^{3},
$$
$$
((f_{0,0}g_{0,3})^{3})^{\rho_1}=(f_{0,1}g_{0,3})^{3},
$$
$$
((g_{0,i}g_{0,i+1})^{3})^{\rho_1}=(g_{0,i}g_{0,i+1})^{3}, \quad i>2.
$$
Conjugating $r_5$ and $\rho_1 r_5 \rho_1$ by $\sigma_1^{-m}$, we get the need relations.
\end{proof}


\subsubsection{Defining relations of $VB_n'$  that follow from the relation $\sigma_i \rho_j = \rho_j \sigma_i$}
Rewrite this relation in the form
$$r_6=\sigma_i \rho_j  \sigma_i^{-1} \rho_j, ~|i - j| >1.$$

In dependence of $i$ and $j$ we will consider the next cases

\medskip
a) $r_6=\sigma_1 \rho_i  \sigma_1^{-1} \rho_i=
S_{1,\sigma_1} S_{\sigma_1,\rho_i} S_{\rho_1,\sigma_1}^{-1}S_{\rho_1,\rho_i}=g_{1,i}a_0^{-1}g_{0,i}$, for $i>2$.

\medskip
b) $r_6=\sigma_2 \rho_j  \sigma_2^{-1} \rho_j=
S_{1,\sigma_2} S_{\sigma_1,\rho_j} S_{\rho_1,\sigma_2}^{-1}S_{\rho_1,\rho_j}=b_{0,0}g_{1,j}b_{0,1}^{-1}g_{0,j}$,
for $j>3$.

\medskip
c) $r_6=\sigma_k \rho_l  \sigma_k^{-1} \rho_l=
S_{1,\sigma_k} S_{\sigma_1,\rho_l} S_{\rho_1,\sigma_k}^{-1}S_{\rho_1,\rho_l}=
c_kg_{1,l}c_k^{-1}g_{0,l}$, $k, l\geq 3$, $|l-k|>1$.

\medskip
d) $r_6=\sigma_i \rho_1  \sigma_i^{-1} \rho_1=
S_{1,\sigma_i} S_{\sigma_1,\rho_1} S_{\rho_1,\sigma_i}^{-1}S_{\rho_1,\rho_1}=c_i c_i^{-1}=1$, for $i>2$.

\medskip
e) $r_6=\sigma_j \rho_2  \sigma_j^{-1} \rho_2=
S_{1,\sigma_j} S_{\sigma_1,\rho_2} S_{\rho_1,\sigma_j}^{-1}S_{\rho_1,\rho_2}=c_jf_{1,0}c_j^{-1}f_{0,1}$, for $j>3$.

Now, the following lemma holds.

\begin{lemma} \label{l10}
From the relation  $r_6$ of $VB_n$, the  following seven types of relations of $VB_n'$ follow:
$$
g_{m+1,i}a_{m}^{-1}g_{m,i}=1, \quad i \geq 3,
$$
$$
a_m g_{m+1,i}g_{m,i}=1, \quad i \geq 3,
$$
$$
b_{m,0}g_{m+1,j}b_{m,1}^{-1}g_{m,j}=1, \quad j \geq 4,
$$
$$
b_{m,1}g_{m+1,j}b_{m,0}^{-1}g_{m,j}=1, \quad j \geq 4,
$$
$$
c_k g_{m+1,l}c_k^{-1}g_{m,l}=1, \quad k, l\geq 3, \quad |l-k|>1,
$$
$$
c_j f_{m+1,0}c_j^{-1}f_{m,1}=1, \quad j \geq 4,
$$
$$
c_j f_{m+1,1}c_j^{-1}f_{m,0}=1, \quad j \geq 4.
$$
\end{lemma}

\begin{proof}
Conjugating $r_6$ by $\rho_1$
and using Lemma \ref{l1}, we get
$$
(g_{1,i}a_0^{-1}g_{0,i})^{\rho_1}=a_0 g_{1,i}g_{0,i}, \quad i \geq 3,
$$
$$
(b_{0,0}g_{1,j}b_{0,1}^{-1}g_{0,j})^{\rho_1}=b_{0,1}g_{1,j}b_{0,0}^{-1}g_{0,j}, \quad j \geq 4,
$$
$$
(c_kg_{1,l}c_k^{-1}g_{0,l})^{\rho_1}=c_kg_{1,l}c_k^{-1}g_{0,l}, \quad k, l\geq 3, \quad |l-k|>1,
$$
$$
(c_jf_{1,0}c_j^{-1}f_{0,1})^{\rho_1}=c_jf_{1,0}c_j^{-1}f_{0,1}, \quad j \geq 4.
$$
Conjugating $r_6$ and $\rho_1 r_6 \rho_1$ by $\sigma_1^{-m}$, we get the need relations.
\end{proof}



\subsubsection{Defining relations of $VB_n'$  that follow from the relation $\rho_i \rho_{i+1} \sigma_i = \sigma_{i+1} \rho_{i} \rho_{i+1}$}
Rewrite this relation in the form
$$r_7=\rho_i \rho_{i+1} \sigma_i \rho_{i+1} \rho_i \sigma_{i+1}^{-1}.$$

Then

\medskip for  $i=1$:
$r_7=\rho_1 \rho_{2} \sigma_1 \rho_{2} \rho_1 \sigma_{2}^{-1}=
S_{1,\rho_1} S_{\rho_1,\rho_2} S_{1,\sigma_1} S_{\sigma_1,\rho_2}S_{\sigma_1\rho_1,\rho_1} S_{1,\sigma_2}^{-1}=
f_{0,1}f_{1,0}b_{0,0}^{-1}$.

\medskip
for  $i=2$:
$r_7=\rho_2 \rho_3 \sigma_2 \rho_3 \rho_2 \sigma_3^{-1}=
S_{1,\rho_2} S_{\rho_1,\rho_3} S_{1,\sigma_2} S_{\sigma_1,\rho_3}S_{\sigma_1\rho_1,\rho_2} S_{1,\sigma_3}^{-1}=
f_{0,0}g_{0,3}b_{0,0}g_{1,3}f_{1,1}c_3^{-1}$.

\medskip
for  $i>2$:
$r_7=\rho_i \rho_{i+1} \sigma_i \rho_{i+1} \rho_i \sigma_{i+1}^{-1}=
S_{1,\rho_i} S_{\rho_1,\rho_{i+1}} S_{1,\sigma_i}
S_{\sigma_1,\rho_{i+1}}S_{\sigma_1\rho_1,\rho_i} S_{1,\sigma_{i+1}}^{-1}=\\
g_{0,i}g_{0,i+1}c_i g_{1,i+1}g_{1,i}c_{i+1}^{-1}$.

Now the following defining relations of  $VB_n'$ follow from the relation
$r_7$.
\begin{lemma} \label{l12}
The following five types of relations in $VB_n'$ follow from the relation $r_7$ of $VB_n$:
$$
f_{m,1}f_{m+1,0}b_{m,0}^{-1}=1,
$$
$$
f_{m,0}a_m f_{m+1,1}b_{m,1}^{-1}=1,
$$
$$
f_{m,0}g_{m,3}b_{m,0}g_{m+1,3}f_{m+1,1}c_3^{-1}=1,
$$
$$
f_{m,1}g_{m,3}b_{m,1}g_{m+1,3}f_{m+1,0}c_3^{-1}=1,
$$
$$
g_{m,i}g_{m,i+1}c_i g_{m+1,i+1}g_{m+1,i}c_{i+1}^{-1}=1, \quad i>2.
$$
\end{lemma}

\begin{proof} Conjugating  $r_7$ by  $\rho_1$
and using Lemma \ref{l1} 2), 4) 5) and 6), we get
$$
(f_{0,1}f_{1,0}b_{0,0}^{-1})^{\rho_1}=f_{0,0}a_0f_{1,1}b_{0,1}^{-1},
$$
$$
(f_{0,0}g_{0,3}b_{0,0}g_{1,3}f_{1,1}c_3^{-1})^{\rho_1}=f_{0,1}g_{0,3}b_{0,1}g_{1,3}f_{1,0}c_3^{-1},
$$
$$
(g_{0,i}g_{0,i+1}c_i g_{1,i+1}g_{1,i}c_{i+1}^{-1})^{\rho_1}=
g_{0,i}g_{0,i+1}c_i g_{1,i+1}g_{1,i}c_{i+1}^{-1}, \quad i>2.
$$
Conjugating $r_7$ and $\rho_1 r_7 \rho_1$ by $\sigma_1^{-m}$, we get the need relations.
\end{proof}

\medskip


Using the relations
$$
f_{m,0} f_{m,1}=1, ~~~\mbox{(see Lemma \ref{l7})},
$$
we can remove elements  $f_{m,1}$, $m\in \mathbb{Z}$, from the generating set and keep only elements

$$
f_{m,0}=f_m,  \quad m\in \mathbb{Z}.
$$

\bigskip

 The following result gives a presentation of $VB_n'$.

\medskip

\begin{theorem}\label{nsp}
The commutator subgroup  $VB_n'$ is generated by elements
$$
a_m, \quad b_{m,\varepsilon},  \quad c_l, \quad f_m, \quad g_{m,l},
$$
where  $m\in \mathbb{Z}$, $\varepsilon=0,1$, $2 < l < n$
and is defined by the relations

$$
b_{m,0} c_j b_{m+1,0}^{-1} c_j^{-1} = 1, \quad j \geq 4,
$$
$$
c_i c_j c_i^{-1} c_j^{-1} = 1, \quad i \geq 3, \quad j > i+1,
$$
$$
b_{m,1} a_m^{-1} c_j a_{m+1} b_{m+1,1}^{-1} c_j^{-1} = 1, \quad j \geq 4,
$$
$$
c_i a_m^{-1} c_j c_i^{-1} a_m c_j^{-1} = 1, \quad i \geq 3, \quad j > i+1.
$$


$$
b_{m+1,0}b_{m+2,0}^{-1}b_{m,0}^{-1}=1,
$$
$$
b_{m,0}c_3b_{m+2,0}c_3^{-1}b_{m+1,0}^{-1}c_3^{-1}=1,
$$
$$
c_i c_{i+1} c_i c_{i+1}^{-1} c_i^{-1} c_{i+1}^{-1}=1, \quad i \geq 3,
$$
$$
a_m b_{m+1,1}a_{m+2}b_{m+2,1}^{-1}a_{m+1}^{-1}b_{m,1}^{-1}=1,
$$
$$
b_{m,1}a_m^{-1}c_3 a_{m+1}b_{m+2,1}a_{m+2}^{-1}a_{m+1}^{-1} c_3^{-1}a_m a_{m+1} b_{m+1,1}^{-1}c_3^{-1}=1,
$$
$$
c_i a_m^{-1}c_{i+1} a_m^{-1}c_i c_{i+1}^{-1} a_m c_i^{-1} a_m c_{i+1}^{-1}=1, \quad i \geq 3.
$$


$$
g_{m,i}^2=1,  \quad i>2.
$$


$$
(f_m g_{m,k})^2=1, \quad k > 3,
$$
$$
(g_{m,i}g_{m,j})^{2}=1, \quad 3\leq i< j\leq n-1, \quad i+1<j.
$$


$$
f_m^{3}=1,
$$
$$
(f_m g_{m,3})^{3}=1,
$$
$$
(g_{m,i}g_{m,i+1})^{3}=1,  \quad i>2,
$$


$$
g_{m+1,i}a_{m}^{-1}g_{m,i}=1, \quad i \geq 3,
$$
$$
b_{m,1}g_{m+1,j}b_{m,0}^{-1}g_{m,j}=1, \quad j \geq 4,
$$
$$
c_k g_{m+1,l}c_k^{-1}g_{m,l}=1, \quad k, l\geq 3, \quad |l-k|>1,
$$
$$
c_j f_{m+1}c_j^{-1}f_{m}^{-1}=1, \quad j \geq 4,
$$


$$
f_{m}^{-1}f_{m+1}b_{m,0}^{-1}=1,
$$
$$
f_{m}a_m f_{m+1}^{-1}b_{m,1}^{-1}=1,
$$
$$
f_{m}g_{m,3}b_{m,0}g_{m+1,3}f_{m+1}^{-1}c_3^{-1}=1,
$$
$$
f_{m}^{-1}g_{m,3}b_{m,1}g_{m+1,3}f_{m+1}c_3^{-1}=1,
$$
$$
g_{m,i}g_{m,i+1}c_i g_{m+1,i+1}g_{m+1,i}c_{i+1}^{-1}=1, \quad i>2.
$$

\end{theorem}

\begin{proof}
The theorem is obtained by combining the  set of relations we have obtained in \lemref{l3}--\lemref{l12}.
\end{proof}
\vspace{0.5cm}

\section{Proof of \thmref{mainth}}\label{pf}
\subsection{Finite generation of $VB_n'$, $n \geq 4$} \label{vb4}
For the next calculations we remove the generators
$b_{m,1}$ and $b_{m,0}$ from the presentation in \thmref{nsp}.

\medskip Using the relations
$$
b_{m,1}=f_{m}a_m f_{m+1}^{-1},
$$
we can remove generators  $b_{m,1}$ from the set of generators.

We have


$$
b_{m,0} c_j b_{m+1,0}^{-1} c_j^{-1} = 1, \quad j > 4,
$$
$$
c_i c_j c_i^{-1} c_j^{-1} = 1, \quad i \geq 3, \quad j \geq i+1,
$$
$$
f_{m} a_m f_{m+1}^{-1} a_m^{-1} c_j a_{m+1} f_{m+2} a_{m+1}^{-1} f_{m+1}^{-1} c_j^{-1} = 1, \quad j \geq 4,
$$
$$
c_i a_m^{-1} c_j c_i^{-1}a_m c_j^{-1} = 1, \quad i \geq 3, \quad j > i+1.
$$


$$
b_{m+1,0}b_{m+2,0}^{-1}b_{m,0}^{-1}=1,
$$
$$
b_{m,0}c_3b_{m+2,0}c_3^{-1}b_{m+1,0}^{-1}c_3^{-1}=1,
$$
$$
c_i c_{i+1} c_i c_{i+1}^{-1} c_i^{-1} c_{i+1}^{-1}=1, \quad i \geq 3,
$$
$$
a_m f_{m+1}a_{m+1} f_{m+2}^{-1}a_{m+2}=f_{m}a_m f_{m+1}^{-1}a_{m+1}f_{m+2}a_{m+2} f_{m+3}^{-1},
$$
$$
f_{m}a_m f_{m+1}^{-1}a_m^{-1}c_3 a_{m+1}f_{m+2}a_{m+2} f_{m+3}^{-1}a_{m+2}^{-1}=
    c_3 f_{m+1}a_{m+1} f_{m+2}^{-1} a_{m+1}^{-1} a_m^{-1} c_3 a_{m+1},
$$
$$
c_i a_m^{-1}c_{i+1} a_m^{-1}c_i c_{i+1}^{-1} a_m c_i^{-1} a_m c_{i+1}^{-1}=1, \quad i \geq 3.
$$


$$
g_{m,i}^2=1,  \quad i>2.
$$


$$
(f_m g_{m,k})^2=1, \quad k > 3,
$$
$$
(g_{m,i}g_{m,j})^{2}=1, \quad 3\leq i< j\leq n-1, \quad i+1<j.
$$


$$
f_m^{3}=1,
$$
$$
(f_m g_{m,3})^{3}=1,
$$
$$
(g_{m,i}g_{m,i+1})^{3}=1,  \quad i>2,
$$


$$
g_{m+1,i}a_{m}^{-1}g_{m,i}=1, \quad i \geq 3,
$$
$$
f_{m}a_m f_{m+1}^{-1}g_{m+1,j}b_{m,0}^{-1}g_{m,j}=1, \quad j \geq 4,
$$
$$
c_k g_{m+1,l}c_k^{-1}g_{m,l}=1, \quad k, l\geq 3, \quad |l-k|>1,
$$
$$
c_j f_{m+1}c_j^{-1}f_{m}^{-1}=1, \quad j \geq 4,
$$


$$
f_{m}^{-1}f_{m+1}b_{m,0}^{-1}=1,
$$
$$
f_{m}g_{m,3}b_{m,0}g_{m+1,3}f_{m+1}^{-1}c_3^{-1}=1,
$$
$$
f_{m}^{-1}g_{m,3}f_{m}a_m f_{m+1}^{-1}g_{m+1,3}f_{m+1}c_3^{-1}=1,
$$
$$
g_{m,i}g_{m,i+1}c_i g_{m+1,i+1}g_{m+1,i}c_{i+1}^{-1}=1, \quad i>2.
$$

Using the relations
$$
b_{m,0}=f_{m}^{-1} f_{m+1},
$$
we can remove the generators  $b_{m,0}$ from the generating set.

\bigskip After removing $b_{m, 0}$ and $b_{m,1}$ we have following set of defining relations of $VB_n'$:


$$
f_{m}^{-1} f_{m+1} c_j = c_j f_{m+1}^{-1} f_{m+2}, \quad j \geq 4,
$$
$$
c_i c_j c_i^{-1} c_j^{-1} = 1, \quad i \geq 3, \quad j > i+1,
$$
$$
f_{m} a_m f_{m+1}^{-1} a_m^{-1} c_j a_{m+1} f_{m+2} a_{m+1}^{-1} f_{m+1}^{-1} c_j^{-1} = 1, \quad j \geq 4,
$$
$$
c_i a_m^{-1} c_j c_i^{-1} a_m c_j^{-1} = 1, \quad i \geq 3, \quad j > i+1.
$$


$$
f_{m} f_{m+1}^{-1} f_{m+2}= f_{m+1} f_{m+2}^{-1} f_{m+3},
$$
$$
f_{m}^{-1} f_{m+1}c_3f_{m+2}^{-1} f_{m+3}=c_3 f_{m+1}^{-1} f_{m+2} c_3,
$$
$$
c_i c_{i+1} c_i c_{i+1}^{-1} c_i^{-1} c_{i+1}^{-1}=1, \quad i \geq 3,
$$
$$
a_m f_{m+1}a_{m+1} f_{m+2}^{-1}a_{m+2}=f_{m}a_m f_{m+1}^{-1}a_{m+1}f_{m+2}a_{m+2} f_{m+3}^{-1},
$$
$$
f_{m}a_m f_{m+1}^{-1}a_m^{-1}c_3 a_{m+1}f_{m+2}a_{m+2} f_{m+3}^{-1}a_{m+2}^{-1}=
    c_3 f_{m+1}a_{m+1} f_{m+2}^{-1} a_{m+1}^{-1} a_m^{-1} c_3 a_{m+1},
$$
$$
c_i a_m^{-1}c_{i+1} a_m^{-1}c_i c_{i+1}^{-1} a_m c_i^{-1} a_m c_{i+1}^{-1}=1, \quad i \geq 3.
$$


$$
g_{m,i}^2=1,  \quad i>2.
$$


$$
(f_m g_{m,k})^2=1, \quad k > 3,
$$
$$
(g_{m,i}g_{m,j})^{2}=1, \quad 3\leq i< j\leq n-1, \quad i+1<j.
$$


$$
f_m^{3}=1,
$$
$$
(f_m g_{m,3})^{3}=1,
$$
$$
(g_{m,i}g_{m,i+1})^{3}=1,  \quad i>2,
$$


$$
g_{m+1,i}a_{m}^{-1}g_{m,i}=1, \quad i \geq 3,
$$
$$
g_{m,j}f_{m}a_m f_{m+1}^{-1}g_{m+1,j}=f_{m}^{-1} f_{m+1}, \quad j \geq 4,
$$
$$
c_k g_{m+1,l}=g_{m,l} c_k, \quad k, l\geq 3, \quad |l-k|>1,
$$
$$
c_j f_{m+1}= f_{m} c_j, \quad j \geq 4,
$$


$$
f_{m}g_{m,3}f_{m}^{-1} f_{m+1}g_{m+1,3}f_{m+1}^{-1}=c_3^{-1},
$$
$$
f_{m}^{-1}g_{m,3}f_{m}a_m f_{m+1}^{-1}g_{m+1,3}f_{m+1}=c_3,
$$
$$
g_{m,i}g_{m,i+1}c_i =c_{i+1} g_{m+1,i} g_{m+1,i+1}, \quad i>2.
$$

\vspace{1cm}

We will use this set of relations to prove that
 $VB_n'$ is finitely generated for all  $n\geq 4$.

\begin{lemma} \label{c1}
The commutator subgroup  $VB_n'$ is finitely generated for all  $n\geq 4$. In particular, $VB_4'$ is generated by 5 elements: $c_3$, $f_0$, $f_1$, $f_2$, $g_{0,3}$, and $VB_n'$, $n \geq 5$, is generated by $2n -3$ elements: $c_3, \ldots, c_{n-1}$, $f_0$, $f_1$, $f_2$, $g_{0,3}, \ldots, g_{0,n-1}$.
\end{lemma}

\begin{proof}

1) Using the relations
$$
g_{m,i}g_{m+1,i}=a_{m}, \quad i \geq 3,
$$
we will remove the generators  $a_m$, $m\in \mathbb{Z}$,
and express them by  $g_{m,i}$, $m\in \mathbb{Z}$, $i\geq 3$.

2) Using the relations
$$
f_{m} f_{m+1}^{-1} f_{m+2}= f_{m+1} f_{m+2}^{-1} f_{m+3},
$$
we can remove the generators $f_m$, for $m\in \mathbb{Z}$, and keep only $f_0$, $f_1$, $f_2$.

3) Using the relations
$$
f_{m}g_{m,3}f_{m}^{-1} f_{m+1}g_{m+1,3}f_{m+1}^{-1}=c_3^{-1},
$$
we can remove the generators $g_{m,3}$, for  $m\in \mathbb{Z}$, and keep only $g_{0,3}$.

If  $n=4$, then we have only generators $g_{0,3}$, $f_0$, $f_1$, $f_2$, $c_3$. Hence,  $VB_4'$ is finitely generated.

If  $n>4$, then

4) Using the relations
$$
g_{m,i}g_{m,i+1}c_i =c_{i+1} g_{m+1,i} g_{m+1,i+1}, \quad i>2.
$$
we can remove the generators $g_{m,i}$, for $m\in \mathbb{Z}$, $i>3$, and keep only $g_{0,i}$.
\end{proof}

\subsection{Infinite generation of $VB_3'$}\label{vb3}

Consider the case  $n = 3$.
 From Theorem \ref{nsp} follows that $VB_3'$ is generated by elements
$$
a_m,\quad b_{m,\varepsilon},\quad f_m,\quad m \in \mathbb{Z},
$$
and is defined by the relations
\begin{equation}\label{5}
b_{m+1,0} b_{m+2,0}^{-1} b_{m,0}^{-1} = 1,
\end{equation}

\begin{equation}\label{6}
a_m b_{m+1,1} a_{m+2} b_{m+2,1}^{-1} a_{m+1}^{-1} b_{m,1}^{-1} = 1,
\end{equation}

\begin{equation}\label{7}
f_{m}^3 = 1,
\end{equation}

\begin{equation}\label{8}
f_{m}^{-1} f_{m+1} f_{m,0}^{-1} = 1,
\end{equation}

\begin{equation}\label{9}
f_{m} a_{m} f_{m+1}^{-1} b_{m,1}^{-1} = 1.
\end{equation}
\\

Now we apply Tietze transformations to the presentation of $VB_3'$. Using relations \eqnref{9}  we can remove the generator $b_{m,1} = f_m a_m f_{m+1}^{-1}$. Then the modified set of defining relations take the form:
\begin{equation}\label{10}
b_{m+1,0} b_{m+2,0}^{-1} b_{m,0}^{-1} = 1,
\end{equation}

\begin{equation}\label{11}
a_m f_{m+1} a_{m+1} f_{m+2}^{-1} a_{m+2} f_{m+3} a_{m+2}^{-1} f_{m+2}^{-1} a_{m+1}^{-1} f_{m+1} a_m^{-1} f_{m}^{-1}= 1,
\end{equation}

\begin{equation}\label{13}
f_{m}^3 = 1,
\end{equation}

\begin{equation}\label{14}
f_{m}^{-1} f_{m+1} b_{m,0}^{-1} = 1,
\end{equation}
\\

Using relations \eqnref{14}  we can remove the generator $b_{m,0} = f_m^{-1} f_{m+1}$. Then $VB_3'$ is generated by elements
$$
a_m,\quad  f_m,\quad m \in \mathbb{Z},
$$
and is defined by relation:
\begin{equation}\label{16}
f_{m+1}^{-1} f_{m+2} f_{m+3}^{-1} f_{m+2} f_{m+1}^{-1} f_m = 1,
\end{equation}

\begin{equation}\label{17}
a_m f_{m+1} a_{m+1} f_{m+2}^{-1} a_{m+2} f_{m+3} a_{m+2}^{-1} f_{m+2}^{-1} a_{m+1}^{-1} f_{m+1} a_m^{-1} f_{m}^{-1}= 1,
\end{equation}

\begin{equation}\label{18}
f_{m}^3 = 1,
\end{equation}
\\

So, we have the following lemma.

\begin{lemma}
The group $VB_3'$ has a presentation with $ \{ a_m$, $f_m$, $m \in \mathbb{Z} \} $ as the generating set, and the relations \eqnref{16}--\eqnref{18} as the defining relations.
\end{lemma}

\begin{lemma}\label{lem200}
$VB_3'$ is not finitely generated.
\end{lemma}

\begin{proof}If we put $f_m=1$, for all $m\in \mathbb{Z}$, then all the relations \eqnref{16}--\eqnref{18} will vanish, i.~e. the subgroup $\left\langle a_m\, | \, m\in \mathbb{Z}  \right\rangle$
is infinitely generated free group with the set of free generators $a_m$, $m\in \mathbb{Z}$
and we have an epimorphism
$$
VB_3' \longrightarrow F_{\infty}=\left\langle a_m,\, \, m\in \mathbb{Z}  \right\rangle
$$
with kernel ${\left\langle f_m,\, \, m\in \mathbb{Z}  \right\rangle}^{VB_3'}$.
\end{proof}

\bigskip

\subsection{Proof of \thmref{mainth}}
Note that $VB_2= F_2 \rtimes S_2$ and hence $VB_2'$ is infinitely generated.
Then the \thmref{mainth} follows by combining \lemref{c1} and \lemref{lem200}.

\bigskip

\subsection{Proof of \corref{cor1}}\label{pfcor1}
In the quotient  $VB_3'/VB_3''$ relations have the form
$$
f_m f_{m+1} = f_{m+2} f_{m+3},
$$
$$
f_{m}^3 = 1.
$$
In the generators  $ f_0, f_1, f_2$, $a_m$, $m\in \mathbb{Z}$, we have relations
$$
f_{0}^3 = f_{1}^3 =f_{2}^3 =1.
$$
Hence,  $VB_3'/VB_3''$ is isomorphic to the direct sum
$$
\mathbb{Z}_3 \oplus \mathbb{Z}_3 \oplus\mathbb{Z}_3 \oplus \mathbb{Z}^{\infty}.
$$


\medskip (2) Consider the case  $n = 4$.
Then $VB_4'$ is generated by elements
$$
a_m, \quad c_3, \quad f_m, \quad g_{m,3},
$$
where  $m\in \mathbb{Z}$, and the defining relations have the form

$$
f_{m} f_{m+1}^{-1} f_{m+2}= f_{m+1} f_{m+2}^{-1} f_{m+3},
$$
$$
f_{m}^{-1} f_{m+1}c_3f_{m+2}^{-1} f_{m+3}=c_3 f_{m+1}^{-1} f_{m+2} c_3,
$$
$$
a_m f_{m+1}a_{m+1} f_{m+2}^{-1}a_{m+2}=f_{m}a_m f_{m+1}^{-1}a_{m+1}f_{m+2}a_{m+2} f_{m+3}^{-1},
$$
$$
f_{m}a_m f_{m+1}^{-1}a_m^{-1}c_3 a_{m+1}f_{m+2}a_{m+2} f_{m+3}^{-1}a_{m+2}^{-1}=
    c_3 f_{m+1}a_{m+1} f_{m+2}^{-1} a_{m+1}^{-1} a_m^{-1} c_3 a_{m+1},
$$
$$
g_{m,3}^2=1,
$$
$$
f_m^{3}=1,
$$
$$
(f_m g_{m,3})^{3}=1,
$$
$$
g_{m+1,3}a_{m}^{-1}g_{m,3}=1,
$$
$$
f_{m}g_{m,3}f_{m}^{-1} f_{m+1}g_{m+1,3}f_{m+1}^{-1}=c_3^{-1},
$$
$$
f_{m}^{-1}g_{m,3}f_{m}a_m f_{m+1}^{-1}g_{m+1,3}f_{m+1}=c_3,
$$

\vspace{1cm}

Consider these relations in the quotient  $VB_4'/VB_4''$ and for the images of the generators
$$
c_3, \quad a_m, \quad f_m, \quad g_{m,3}, \quad m\in \mathbb{Z},
$$
we will use the same symbols.

From relation  $ g_{m,3}^2=f_m^{3}=(f_m g_{m,3})^{3}=1 $ we get $g_{m,3}=1$.

Then from the relation  $g_{m+1,3}a_{m}^{-1}g_{m,3}=1$ follows that  $a_{m}=1$.

The other relations have the form
$$
c_3=f_m^{3}=1, \quad f_{m} f_{m+1} =  f_{m+2} f_{m+3}.
$$
Hence, we can keep only generators  $f_0$, $f_1$, $f_2$ and defining relations
 $f_0^3=f_1^3=f_2^3=1$.

Hence the quotient  $VB_4'/VB_4''$ is isomorphic to the direct product
$\mathbb{Z}_3 \oplus \mathbb{Z}_3 \oplus\mathbb{Z}_3 $
of three cyclic groups  $\mathbb{Z}_3  $
of order 3.


\medskip Consider the case  $n > 4$.  We will consider relations of  $VB_n'$ in the quotient  $VB_n'/VB_n''$
and will denote the images of the generators
$$
a_m, \quad b_{m,\varepsilon},  \quad c_l, \quad f_m, \quad g_{m,l},
$$
where  $m\in \mathbb{Z}$, $\varepsilon=0,1$, $2<l<n$
by the same symbols.

As in the case  $n=4$ we get $ g_{m,3}=a_{m}=1$.

Then from the relations
$$
g_{m,i}^2=1,  \quad  (g_{m,i}g_{m,i+1})^{3}=1,  \quad i>2,
$$
follows that  $ g_{m,i}=1$, $i>2$.

From the relations
$$
f_m^{3}=1, \quad (f_m g_{m,k})^2=1, \quad k > 3,
$$
follows that  $ f_m=1$.

Remaining relations have the form
$$
c_i=1, \quad i \geq 3.
$$
This completes the proof.

\section{Commutator subgroup of the welded braid group} \label{wel}

The welded braid group $WB_n$, $n \geq 2$, is the quotient of $VB_n$ by the relations
$$
\rho_i \sigma_{i+1} \sigma_i = \sigma_{i+1} \sigma_i \rho_{i+1},~~i = 1, 2, \ldots, n-2.
$$
In this section we will find a presentation of $WB_n'$. We will use the same set of generators that we used for $VB_n$ and $VB_n'$. Hence to find defining relations for $WB_n'$ we need to add relations that follow from the relation
$$
r_8 = \rho_i \sigma_{i+1} \sigma_i \rho_{i+1} \sigma_{i}^{-1} \sigma_{i+1}^{-1}.
$$
Depending on $i$ we will consider 3 cases:

if $i = 1$, then
$$
r_8 = \rho_1 \sigma_{2} \sigma_1 \rho_{2} \sigma_{1}^{-1} \sigma_{2}^{-1} = S_{1,\rho_1} S_{\rho_1,\sigma_2} S_{\sigma_1 \rho_1,\sigma_1} S_{\sigma_1^2 \rho_1,\rho_2} S_{\sigma_1,\sigma_1}^{-1} S_{1,\sigma_2}^{-1} = b_{0,1} a_1 f_{2,1} b_{0,0}^{-1};
$$

if $i = 2$, then
$$
r_8 = \rho_2 \sigma_{3} \sigma_2 \rho_{3} \sigma_{2}^{-1} \sigma_{3}^{-1} = S_{1,\rho_2} S_{\rho_1,\sigma_3} S_{\sigma_1 \rho_1,\sigma_2} S_{\sigma_1^2 \rho_1,\rho_3} S_{\sigma_1,\sigma_2}^{-1} S_{1,\sigma_3}^{-1} = f_{0,0} c_3 b_{1,1} g_{2,3} b_{1,0}^{-1} c_{3}^{-1};
$$

if $i > 2$, then
\begin{eqnarray*}
r_8 &=& \rho_i \sigma_{i+1} \sigma_i \rho_{i+1} \sigma_{i}^{-1} \sigma_{i+1}^{-1} \\
&=& S_{1,\rho_i} S_{\rho_1,\sigma_{i+1}} S_{\sigma_1 \rho_1,\sigma_i} S_{\sigma_1^2 \rho_1,\rho_{i+1}} S_{\sigma_1,\sigma_i}^{-1} S_{1,\sigma_{i+1}}^{-1} \\&=& g_{0,i} c_{i+1} c_i g_{2,i+1} c_{i}^{-1} c_{i+1}^{-1}.
\end{eqnarray*}

We will use the following conjugation rules

\begin{lemma} \label{con}
In $WB_n$ the following conjugation rules hold:
\begin{itemize}
\item[(1)] $a_1^{\rho_1} = a_0 a_1^{-1} a_0^{-1},$

\item[(2)] $f_{2,1}^{\rho_1} = a_0 a_1 f_{2,0} a_1^{-1} a_0^{-1},$

\item[(3)] $g_{2,i}^{\rho_1} = a_0 a_1 g_{2,i} a_1^{-1} a_0^{-1}$ for $i > 2$.\end{itemize}
\end{lemma}

\begin{proof}
(1) Note that \begin{eqnarray*}\rho_1 a_1 \rho_1 &=&  \rho_1 \sigma_1 \rho_1 \sigma_1 \rho_1 \sigma_1^{-1} \sigma_1^{-1} \rho_1\\
&=& S_{1,\rho_1} S_{\rho_1,\sigma_{1}} S_{\sigma_1 \rho_1,\rho_1} S_{\sigma_1,\sigma_{1}} S_{\sigma_1^2,\rho_1} S_{\sigma_1 \rho_1,\sigma_1}^{-1} S_{\rho_1,\sigma_{1}}^{-1} S_{\rho_1,\rho_1} \\& = & a_0 a_1^{-1} a_0^{-1}; \end{eqnarray*}

(2) Next we have,  \begin{eqnarray*} \rho_1 f_{2,1} \rho_1 &=& \rho_1 \sigma_1 \sigma_1 \rho_1 \rho_2  \sigma_1^{-1} \sigma_1^{-1} \rho_1 \\&=&
S_{1,\rho_1} S_{\rho_1,\sigma_{1}} S_{\sigma_1 \rho_1, \sigma_1} S_{\sigma_1^2 \rho_1,\rho_{1}} S_{\sigma_1^2,\rho_2} S_{\sigma_1 \rho_1,\sigma_1}^{-1} S_{\rho_1,\sigma_{1}}^{-1} S_{\rho_1,\rho_1} \\&=&  a_0 a_1 f_{2,0} a_1^{-1} a_0^{-1};\end{eqnarray*}

(3) Finally, \begin{eqnarray*} \rho_1 g_{2,i} \rho_1 &=& \rho_1 \sigma_1 \sigma_1 \rho_i \rho_1  \sigma_1^{-1} \sigma_1^{-1} \rho_1 \\
&=& S_{1,\rho_1} S_{\rho_1,\sigma_{1}} S_{\sigma_1 \rho_1, \sigma_1} S_{\sigma_1^2 \rho_1,\rho_{i}} S_{\sigma_1^2,\rho_1} S_{\sigma_1 \rho_1,\sigma_1}^{-1} S_{\rho_1,\sigma_{1}}^{-1} S_{\rho_1,\rho_1} \\ &=& a_0 a_1 g_{2,i} a_1^{-1} a_0^{-1}.\end{eqnarray*}
This proves the lemma.
\end{proof}

\begin{lemma} \label{l5.2}
From the relation  $r_8$  of $WB_n$, the following six types of relations of $WB_n'$ follow:

$b_{m,1} a_{m+1} f_{m+2,1} b_{m,0}^{-1} = 1$,

$f_{m,0} c_3 b_{m+1,1} g_{m+2,3} b_{m+1,0}^{-1} c_3^{-1} = 1$,

$g_{m,i} c_{i+1} c_i g_{m+2,i+1} c_i^{-1} c_{i+1}^{-1} = 1$,

$b_{m,0}  f_{m+2,0} a_{m+1}^{-1}  b_{m,1}^{-1} = 1$,

$f_{m,1} c_3 f_{m+1,0} a_{m+1} g_{m+2,3} b_{m+1,1}^{-1} c_3^{-1} = 1$,

$g_{m,i} c_{i+1} a_m^{-1} c_i a_{m+1} g_{m+2,i+1} a_{m+1}^{-1} c_i^{-1} a_m c_{i+1}^{-1} = 1$.

\end{lemma}

\begin{proof}
Conjugating relations $r_8$ by $\rho_1$ and using Lemma \ref{con}, we get 3 relations:
$$
(b_{0,1} a_1 f_{2,1} b_{0,0}^{-1})^{\rho_1} = b_{0,0} f_{2,0} a_1^{-1}  b_{0,1}^{-1},
$$

$$
(f_{0,0} c_3 b_{1,1} g_{2,3} b_{1,0}^{-1} c_{3}^{-1})^{\rho_1} = f_{0,1} c_3 f_{1,0} a_1 g_{2,3} b_{1,1}^{-1} c_{3}^{-1},
$$

$$
(g_{0,i} c_{i+1} c_i g_{2,i+1} c_{i}^{-1} c_{i+1}^{-1})^{\rho_1} = g_{0,i} c_{i+1} a_0^{-1} c_i a_1 g_{2,i+1} a_1^{-1} c_{i}^{-1} a_0 c_{i+1}^{-1}.
$$
Conjugating relations $r_8$ and $\rho_1 r_8 \rho_1$ by $\sigma_1^{-m}$, we get the six  relations from the lemma.
\end{proof}

Thus we have the following.
\begin{cor}
The commutator subgroup  $WB_n'$ is generated by elements
$$
a_m, \quad b_{m,\varepsilon},  \quad c_l, \quad f_m, \quad g_{m,l},
$$
where  $m\in \mathbb{Z}$, $\varepsilon=0,1$, $2 < l < n$
and is defined by the relations in \thmref{nsp} and \lemref{l5.2}.
\end{cor}
\subsection{Presentation of $WB_3'$}
We have found a presentation of $VB_3'$. To get a presentation of $WB_3'$ we need to add two series of relations:
\begin{equation}\label{19}
b_{m,1} a_{m+1} f_{m+2,1} b_{m,0}^{-1} = 1,
\end{equation}
\begin{equation}\label{20}
b_{m,0} f_{m+2,0} a_{m+1}^{-1} b_{m,1}^{-1} = 1,
\end{equation}
that follow from Lemma \ref{l5.2}.

As in the case of $VB_3'$ we can remove the generator $f_{m,1}$, using the relation $f_{m,0} f_{m,1} = 1$. Then the relations (\ref{19})--(\ref{20}) have the form
\begin{equation}\label{21}
b_{m,1} a_{m+1} f_{m+2}^{-1} b_{m,0}^{-1} = 1,
\end{equation}
\begin{equation}\label{22}
b_{m,0} f_{m+2} a_{m+1}^{-1} b_{m,1}^{-1} = 1,
\end{equation}
where we denote $f_m = f_{m,0}$.

Using the relations
$$
b_{m,1} = f_m a_m f_{m+1}^{-1},
$$
that hold in $VB_3'$, we can remove $b_{m,1}$. Then the relations (\ref{21})--(\ref{22}) have the form
\begin{equation}\label{23}
f_{m} a_{m} f_{m+1}^{-1} a_{m+1} f_{m+2}^{-1} b_{m,0}^{-1} = 1,
\end{equation}
\begin{equation}\label{24}
b_{m,0} f_{m+2} a_{m+1}^{-1} f_{m+1} a_m^{-1} f_{m}^{-1} = 1,
\end{equation}
We see that the second relation is inverse to the first one. Hence, we can remove the second relation.

Next, using the relations $f_m^{-1} f_{m+1} b_{m,0}^{-1} = 1$, which hold in $VB_3'$, we can remove the generator $b_{m,0}$. Then (\ref{23}) has the form
\begin{equation}\label{25}
f_{m} a_{m} f_{m+1}^{-1} a_{m+1} f_{m+2}^{-1} f_{m+1}^{-1} f_m = 1.
\end{equation}
Using the presentation of $VB_3'$ we get

\begin{prop}
The group $WB_3'$ is  generated by elements
$$
a_m,\quad  f_m,\quad m \in \mathbb{Z},
$$
and is defined by relation:
\begin{equation}\label{26}
f_{m+1}^{-1} f_{m+2} f_{m+3}^{-1} f_{m+2} f_{m+1}^{-1} f_m = 1,
\end{equation}

\begin{equation}\label{27}
a_m f_{m+1} a_{m+1} f_{m+2}^{-1} a_{m+2} f_{m+3} a_{m+2}^{-1} f_{m+2}^{-1} a_{m+1}^{-1} f_{m+1} a_m^{-1} f_{m}^{-1}= 1,
\end{equation}

\begin{equation}\label{28}
f_{m}^3 = 1,
\end{equation}

\begin{equation}\label{29}
a_{m} f_{m+1}^{-1} a_{m+1} f_{m+2}^{-1} f_{m+1}^{-1} f_m^{-1} = 1.
\end{equation}
\end{prop}

\bigskip  As consequence we get

 \begin{cor}\label{wb1}
$WB_3'$ is generated by $a_0, f_0, f_1, f_2$.
 \end{cor}

\begin{proof}
From the set of relations (\ref{26}) we can express the generators $f_k$, where $k > 2$ or $k < 0$, as words in the generators $f_0, f_1, f_2$ and analogously,
from the set of relations (\ref{29}) we can express the generators $a_l$, where $l \not= 0$, as words in the generators $a_0, f_0, f_1, f_2$.
\end{proof}

\begin{cor}\label{wb11}
 $WB_3'/WB_3''$ is isomorphic to the direct sum
$$
\mathbb{Z}_3 \oplus \mathbb{Z}_3 \oplus \mathbb{Z}_3 \oplus \mathbb{Z}.
$$
\end{cor}
\begin{proof}
In the quotient  $WB_3'/WB_3''$ the relations have the form
$$
f_m f_{m+1} = f_{m+2} f_{m+3},
$$
$$
f_{m}^3 = 1,
$$
$$
a_{m}  a_{m+1} = f_m f_{m+1}^{-1} f_{m+2}.
$$
In the generators  $a_0, f_0, f_1, f_2$ we have relations
$$
f_{0}^3 = f_{1}^3 =f_{2}^3 =1.
$$
This completes the proof. \end{proof}


\subsection{The commutator subgroup $WB_4'$}

In $WB_4'$ we have relations of  $VB_4'$ and the following relations:
$$
b_{m,1} a_{m+1} f_{m+2,1} b_{m,0}^{-1} = 1,
$$
$$
b_{m,0}  f_{m+2,0} a_{m+1}^{-1}  b_{m,1}^{-1} = 1,
$$
$$
f_{m,0} c_3 b_{m+1,1} g_{m+2,3} b_{m+1,0}^{-1} c_3^{-1} = 1,
$$
$$
f_{m,1} c_3 f_{m+1,0} a_{m+1} g_{m+2,3} b_{m+1,1}^{-1} c_3^{-1} = 1.
$$
Excluding the generators
$$
b_{m,0}=f_{m}^{-1}f_{m+1},\quad
b_{m,1}=f_{m} a_{m}f_{m+1}^{-1},\quad
f_{m,1}=f_{m,0}^{-1}=f_{m}^{-1}
$$
from these relations.
We get relations
$$
f_{m} a_{m}f_{m+1}^{-1} a_{m+1} f_{m+2}^{-1} f_{m+1}^{-1} f_{m}= 1,
$$
$$
f_{m}^{-1}f_{m+1}  f_{m+2} a_{m+1}^{-1} f_{m+1} a_{m}^{-1} f_{m}^{-1}  = 1,
$$
$$
f_{m} c_3 f_{m+1} a_{m+1} f_{m+2}^{-1} g_{m+2,3}  f_{m+2}^{-1} f_{m+1} c_3^{-1} = 1,
$$
$$
f_{m}^{-1} c_3 f_{m+1} a_{m+1} g_{m+2,3} f_{m+2} a_{m+1}^{-1} f_{m+1}^{-1} c_3^{-1} = 1.
$$
The second relation is inverse of the first relation. Hence, we can keep only the first relation. Rewrite it in the form
$$
a_{m}f_{m+1}^{-1} a_{m+1} = f_{m} f_{m+1} f_{m+2} .
$$
Rewrite the third and the forth relations in the form
$$
c_3 f_{m+1} a_{m+1} f_{m+2}^{-1} g_{m+2,3}  f_{m+2}^{-1} f_{m+1} c_3^{-1}f_{m}  = 1,
$$
$$
 c_3 f_{m+1} a_{m+1} g_{m+2,3} f_{m+2} a_{m+1}^{-1} f_{m+1}^{-1} c_3^{-1}f_{m}^{-1} = 1.
$$
From these relations:
$$
f_{m+2}^{-1} g_{m+2,3}  f_{m+2}^{-1} f_{m+1} c_3^{-1}f_{m}  =
 g_{m+2,3} f_{m+2} a_{m+1}^{-1} f_{m+1}^{-1} c_3^{-1}f_{m}^{-1}.
$$
Since in $VB_4'$ holds
$$
(g_{m+2,3}  f_{m+2})^3=1,\quad g_{m+2,3}^2=1,
$$
then
$$
(g_{m+2,3}  f_{m+2})^{-2}  f_{m+2}^{-1} f_{m+1} c_3^{-1}f_{m}  =
  a_{m+1}^{-1} f_{m+1}^{-1} c_3^{-1}f_{m}^{-1}
$$
and
$$
g_{m+2,3}  f_{m+2}  f_{m+2}^{-1} f_{m+1} c_3^{-1}f_{m}  =
  a_{m+1}^{-1} f_{m+1}^{-1} c_3^{-1}f_{m}^{-1}.
$$
Therefore,
$$
g_{m+2,3}   =
  a_{m+1}^{-1} f_{m+1}^{-1} c_3^{-1}f_{m} c_3 f_{m+1}^{-1}.
$$
Including this expression of  $g_{m+2,3}$ in the forth relation:
$$
f_{m}^{-1} c_3 f_{m+1} a_{m+1} a_{m+1}^{-1} f_{m+1}^{-1} c_3^{-1}f_{m}
c_3 f_{m+1}^{-1} f_{m+2} a_{m+1}^{-1} f_{m+1}^{-1} c_3^{-1} = 1.
$$
We get after cancelation
$$
a_{m+1}=f_{m+1} f_{m+2}.
$$
Including this expression of $a_{m+1}$ in the expression for
$g_{m+2,3}$, we get
$$
g_{m+2,3} =  f_{m+2}^{-1} f_{m+1}^{-2} c_3^{-1}f_{m} c_3 f_{m+1}^{-1}
$$
or
$$
g_{m+2,3} =  f_{m+2}^{-1} f_{m+1} c_3^{-1}f_{m} c_3 f_{m+1}^{-1}.
$$
Next, the relation  $a_{m}f_{m+1}^{-1} a_{m+1} = f_{m} f_{m+1} f_{m+2}$
after substitution $a_{m}=f_{m} f_{m+1}$, $a_{m+1}=f_{m+1} f_{m+2}$
becomes an identity.

Hence, the new relations in  $WB_4'$ are equal to relations
$$
g_{m+2,3} =  f_{m+2}^{-1} f_{m+1} c_3^{-1}f_{m} c_3 f_{m+1}^{-1},
$$
$$
a_{m}=f_{m} f_{m+1}.
$$

The full set of relations in $WB_4'$ has the form:
$$
f_{m} f_{m+1}^{-1} f_{m+2}= f_{m+1} f_{m+2}^{-1} f_{m+3},
$$
$$
f_{m}^{-1} f_{m+1}c_3f_{m+2}^{-1} f_{m+3}=c_3 f_{m+1}^{-1} f_{m+2} c_3,
$$
$$
a_m f_{m+1}a_{m+1} f_{m+2}^{-1}a_{m+2}=f_{m}a_m f_{m+1}^{-1}a_{m+1}f_{m+2}a_{m+2} f_{m+3}^{-1},
$$
$$
f_{m}a_m f_{m+1}^{-1}a_m^{-1}c_3 a_{m+1}f_{m+2}a_{m+2} f_{m+3}^{-1}a_{m+2}^{-1}=
    c_3 f_{m+1}a_{m+1} f_{m+2}^{-1} a_{m+1}^{-1} a_m^{-1} c_3 a_{m+1},
$$
$$
g_{m,3}^2=1,
$$
$$
f_m^{3}=1,
$$
$$
(f_m g_{m,3})^{3}=1,
$$
$$
g_{m+1,3}a_{m}^{-1}g_{m,3}=1,
$$
$$
f_{m}g_{m,3}f_{m}^{-1} f_{m+1}g_{m+1,3}f_{m+1}^{-1}=c_3^{-1},
$$
$$
f_{m}^{-1}g_{m,3}f_{m}a_m f_{m+1}^{-1}g_{m+1,3}f_{m+1}=c_3,
$$
$$
g_{m+2,3} =  f_{m+2}^{-1} f_{m+1} c_3^{-1}f_{m} c_3 f_{m+1}^{-1},
$$
$$
a_{m}=f_{m} f_{m+1}.
$$

Transform these relations, excluding
 $a_m$ and $g_{m,3}$.

1) The relation
$a_m f_{m+1}a_{m+1} f_{m+2}^{-1}a_{m+2}=f_{m}a_m f_{m+1}^{-1}a_{m+1}f_{m+2}a_{m+2} f_{m+3}^{-1}$,
after substitution
$$
a_{m}=f_{m} f_{m+1},\quad a_{m+1}=f_{m+1} f_{m+2},\quad a_{m+2}=f_{m+2} f_{m+3}
$$
has the form
$$
f_{m} f_{m+1} f_{m+1}f_{m+1} f_{m+2} f_{m+2}^{-1}f_{m+2} f_{m+3}=
$$
$$
=f_{m}f_{m} f_{m+1} f_{m+1}^{-1}f_{m+1} f_{m+2}f_{m+2}f_{m+2} f_{m+3} f_{m+3}^{-1}.
$$
Using the relation  $f_{m}^3=1$, we get
$$
f_{m} f_{m+1} =f_{m+2} f_{m+3}.
$$

2) The relation
$$
f_{m}a_m f_{m+1}^{-1}a_m^{-1}c_3 a_{m+1}f_{m+2}a_{m+2} f_{m+3}^{-1}a_{m+2}^{-1}=
    c_3 f_{m+1}a_{m+1} f_{m+2}^{-1} a_{m+1}^{-1} a_m^{-1} c_3 a_{m+1},
$$
after substitution
$$
a_{m}=f_{m} f_{m+1},\quad a_{m+1}=f_{m+1} f_{m+2},\quad a_{m+2}=f_{m+2} f_{m+3}
$$
has the form
$$
f_{m}f_{m} f_{m+1} f_{m+1}^{-1}f_{m+1}^{-1}f_{m}^{-1} c_3 f_{m+1} f_{m+2}f_{m+2}
f_{m+2} f_{m+3} f_{m+3}^{-1} f_{m+3}^{-1}f_{m+2}^{-1}=
$$
$$
  =c_3 f_{m+1}f_{m+1} f_{m+2} f_{m+2}^{-1} f_{m+2}^{-1}f_{m+1}^{-1}
  f_{m+1}^{-1}f_{m}^{-1} c_3 f_{m+1} f_{m+2},
$$
or, after cancelation and using the relation  $f_{m}^3=1$ we get
$$
f_{m}^{-1} f_{m+1}^{-1} f_{m}^{-1} c_3 f_{m+1} f_{m+3}^{-1}f_{m+2}=
 c_3 f_{m+1}^{-1} f_{m+2}^{-1} f_{m+1}f_{m}^{-1} c_3 f_{m+1}.
$$

3)  The relation $g_{m+2,3}^2=1$ after substitution
$$
g_{m+2,3} =  f_{m+2}^{-1} f_{m+1} c_3^{-1}f_{m} c_3 f_{m+1}^{-1},
$$
has the form
$$
f_{m+2}^{-1} f_{m+1} c_3^{-1}f_{m} c_3 f_{m+1}^{-1}f_{m+2}^{-1} f_{m+1} c_3^{-1}f_{m} c_3 f_{m+1}^{-1}=1.
$$

4)  The relation $(f_m g_{m,3})^{3}=1$ after substitution
$$
g_{m,3} =  f_{m}^{-1} f_{m-1} c_3^{-1}f_{m-2} c_3 f_{m-1}^{-1},
$$
has the form
$$
(f_m f_{m}^{-1} f_{m-1} c_3^{-1}f_{m-2} c_3 f_{m-1}^{-1})^{3}=1
$$
and is identity since  $f_{m}^3=1$.

5) The relation $g_{m+1,3}a_{m}^{-1}g_{m,3}=1$ after substitution
$$
g_{m+1,3} =  f_{m+1}^{-1} f_{m} c_3^{-1}f_{m-1} c_3 f_{m}^{-1}, \quad
g_{m,3} =  f_{m}^{-1} f_{m-1} c_3^{-1}f_{m-2} c_3 f_{m-1}^{-1}, \quad
a_{m}=f_{m} f_{m+1}
$$
has the form
$$
f_{m+1}^{-1} f_{m} c_3^{-1}f_{m-1} c_3 f_{m}^{-1} f_{m+1}^{-1}f_{m}^{-1}
f_{m}^{-1} f_{m-1} c_3^{-1}f_{m-2} c_3 f_{m-1}^{-1}=1.
$$
Using the relation  $f_{m}^3=1$ and changing the index  $m$ on $m+1$, we get
$$
f_{m+2}^{-1} f_{m+1} c_3^{-1}f_{m} c_3 f_{m+1}^{-1} f_{m+2}^{-1}f_{m+1}
f_{m} c_3^{-1}f_{m-1} c_3 f_{m}^{-1}=1.
$$

6) The relation $f_{m}g_{m,3}f_{m}^{-1} f_{m+1}g_{m+1,3}f_{m+1}^{-1}=c_3^{-1}$
after substitution
$$
g_{m+1,3} =  f_{m+1}^{-1} f_{m} c_3^{-1}f_{m-1} c_3 f_{m}^{-1}, \quad
g_{m,3} =  f_{m}^{-1} f_{m-1} c_3^{-1}f_{m-2} c_3 f_{m-1}^{-1}, \quad
$$
has the form
$$
f_{m}f_{m}^{-1} f_{m-1} c_3^{-1}f_{m-2} c_3 f_{m-1}^{-1}f_{m}^{-1} f_{m+1}
f_{m+1}^{-1} f_{m} c_3^{-1}f_{m-1} c_3 f_{m}^{-1}f_{m+1}^{-1}=c_3^{-1}
$$
or after cancelation
$$
f_{m-1} c_3^{-1}f_{m-2} c_3 f_{m-1}^{-1}c_3^{-1}f_{m-1} c_3 f_{m}^{-1}f_{m+1}^{-1}=c_3^{-1}.
$$

7) The relation $f_{m}^{-1}g_{m,3}f_{m}a_m f_{m+1}^{-1}g_{m+1,3}f_{m+1}=c_3$
after substitution
$$
g_{m+1,3} =  f_{m+1}^{-1} f_{m} c_3^{-1}f_{m-1} c_3 f_{m}^{-1}, \quad
g_{m,3} =  f_{m}^{-1} f_{m-1} c_3^{-1}f_{m-2} c_3 f_{m-1}^{-1}, \quad
a_{m}=f_{m} f_{m+1}
$$
has the form
$$
f_{m}^{-1}f_{m}^{-1} f_{m-1} c_3^{-1}f_{m-2} c_3 f_{m-1}^{-1}f_{m}
f_{m} f_{m+1} f_{m+1}^{-1}f_{m+1}^{-1} f_{m} c_3^{-1}f_{m-1} c_3 f_{m}^{-1}f_{m+1}=c_3
$$
or, after cancelation and using the relation $f_{m}^3=1$ we get
$$
f_{m} f_{m-1} c_3^{-1}f_{m-2} c_3 f_{m-1}^{-1}f_{m}^{-1} f_{m+1}^{-1} f_{m}
c_3^{-1}f_{m-1} c_3 f_{m}^{-1}f_{m+1}=c_3.
$$

Hence, we have proven

\begin{theorem} \label{wb3}
The group  $WB_4'$ is generated by  $c_3$, $f_m$, $m \in \mathbb{Z}$,
and is defined by the relations
$$
f_{m} f_{m+1}^{-1} f_{m+2}= f_{m+1} f_{m+2}^{-1} f_{m+3},
$$
$$
f_{m}^{-1} f_{m+1}c_3f_{m+2}^{-1} f_{m+3}=c_3 f_{m+1}^{-1} f_{m+2} c_3,
$$
$$
f_{m} f_{m+1} =f_{m+2} f_{m+3},
$$
$$
f_{m}^{-1} f_{m+1}^{-1} f_{m}^{-1} c_3 f_{m+1} f_{m+3}^{-1}f_{m+2}=
 c_3 f_{m+1}^{-1} f_{m+2}^{-1} f_{m+1}f_{m}^{-1} c_3 f_{m+1}.
$$
$$
f_{m+2}^{-1} f_{m+1} c_3^{-1}f_{m} c_3 f_{m+1}^{-1}f_{m+2}^{-1} f_{m+1} c_3^{-1}f_{m} c_3 f_{m+1}^{-1}=1,
$$
$$
f_m^{3}=1,
$$
$$
f_{m+2}^{-1} f_{m+1} c_3^{-1}f_{m} c_3 f_{m+1}^{-1} f_{m+2}^{-1}f_{m+1}
f_{m} c_3^{-1}f_{m-1} c_3 f_{m}^{-1}=1,
$$
$$
f_{m-1} c_3^{-1}f_{m-2} c_3 f_{m-1}^{-1}c_3^{-1}f_{m-1} c_3 f_{m}^{-1}f_{m+1}^{-1}=c_3^{-1},
$$
$$
f_{m} f_{m-1} c_3^{-1}f_{m-2} c_3 f_{m-1}^{-1}f_{m}^{-1} f_{m+1}^{-1} f_{m}
c_3^{-1}f_{m-1} c_3 f_{m}^{-1}f_{m+1}=c_3.
$$
\end{theorem}

\begin{cor}\label{wb2}
The group $WB_4'$ is generated by  $c_3$, $f_0$, $f_1$, $f_2$.
\end{cor}

Indeed, using the relations
$$
f_{m} f_{m+1}^{-1} f_{m+2}= f_{m+1} f_{m+2}^{-1} f_{m+3},
$$
we can save from the generators $f_m$, $m \in \mathbb{Z}$, only the relations
$f_0$, $f_1$, $f_2$.

\begin{cor} \label{wb22}
$WB_4'/WB_4''\cong \mathbb{Z}_3$.
\end{cor}

Indeed, considering relations of $WB_4'$ by modulo $WB_4''$
we see that $f_mf_{m+1}=1$, $f_m^{3}=1$ and  $c_3=1$.

\subsection{The commutator subgroup $WB_n'$ for $n \geq 5$}

\begin{theorem} \label{wb31}
The group $WB_n'$, $n \geq 5$,  is generated by $n$ elements
$f_0$, $f_1$, $f_2$,  $c_3$, $\ldots$, $c_{n-1}$.
\end{theorem}

\begin{proof}
As we proved before,
 $VB_n'$, $n \geq 5$, is generated by elements
$c_3, \ldots, c_{n-1}$, $f_0$, $f_1$, $f_2$, $g_{0,3}, \ldots, g_{0,n-1}$.

The group $WB_n'$, $n \geq 5$,  is defined by relations of  $VB_n'$ and the relations:
$$
b_{m,1} a_{m+1} f_{m+2,1} b_{m,0}^{-1} = 1,
$$
$$
b_{m,0}  f_{m+2,0} a_{m+1}^{-1}  b_{m,1}^{-1} = 1,
$$
$$
f_{m,0} c_3 b_{m+1,1} g_{m+2,3} b_{m+1,0}^{-1} c_3^{-1} = 1,
$$
$$
f_{m,1} c_3 f_{m+1,0} a_{m+1} g_{m+2,3} b_{m+1,1}^{-1} c_3^{-1} = 1.
$$
$$
g_{m,i} c_{i+1} c_i g_{m+2,i+1} c_i^{-1} c_{i+1}^{-1} = 1,
$$
$$
g_{m,i} c_{i+1} a_m^{-1} c_i a_{m+1} g_{m+2,i+1} a_{m+1}^{-1} c_i^{-1} a_m c_{i+1}^{-1} = 1.
$$

Similar to the group $WB_4'$, the firs for relations are equivalent to the relations
$$
g_{m+2,3} =  f_{m+2}^{-1} f_{m+1} c_3^{-1}f_{m} c_3 f_{m+1}^{-1},\quad
a_{m}=f_{m} f_{m+1}.
$$

Hence, the additional relations of  $WB_n'$, $\geq 5$
have the form
$$
g_{m+2,3} =  f_{m+2}^{-1} f_{m+1} c_3^{-1}f_{m} c_3 f_{m+1}^{-1},
$$
$$
a_{m}=f_{m} f_{m+1},
$$
$$
g_{m,i} c_{i+1} c_i g_{m+2,i+1} c_i^{-1} c_{i+1}^{-1} = 1,
$$
$$
g_{m,i} c_{i+1} f_{m+1}^{-1}f_{m}^{-1} c_i f_{m+1} f_{m+2}
g_{m+2,i+1} f_{m+2}^{-1}f_{m+1}^{-1}  c_i^{-1} f_{m} f_{m+1} c_{i+1}^{-1} = 1.
$$
Using the relations
$g_{m,i} c_{i+1} c_i g_{m+2,i+1} c_i^{-1} c_{i+1}^{-1} = 1$,
we can express the generators $g_{m,i}$, $i\geq 4$, as words in the generators
$c_3,$ $\ldots$, $c_{n-1}$, $f_m$, $g_{m,3}$, $m \in \mathbb{Z} $.
Also, as in the case of the group $WB_4'$, we can express the generators  $f_m$, $g_{m,3}$, $m \in \mathbb{Z} $,
as words in the generators
$c_3$, $f_0$, $f_1$, $f_2$.
\end{proof}

\subsection{Proof of \thmref{wbth}}
\begin{proof}
Parts (1)  of \thmref{wbth} follows by combining \corref{wb1}, \corref{wb2} and  \thmref{wb3}. Part (2) and (3) follow from  \corref{wb11}, and \corref{wb22}.

For $n \geq 5$, note that $WB_n'$
is perfect as a quotient of the perfect group  $VB_n'$. This proves (4).
\end{proof}

\end{document}